\documentclass[12pt]{article}

\usepackage[left=2.5cm, right=2.5cm, top=2cm, bottom=2cm]{geometry}

\usepackage{bbm}
\usepackage{amsmath, amsthm, amssymb}
\usepackage{amsfonts}
\usepackage{microtype}
\usepackage{pdfrender}
\usepackage[ansinew]{inputenc}
\usepackage[dvips]{epsfig}
\usepackage{graphicx}
\usepackage[english]{babel}

\usepackage{cite}
\usepackage{graphicx}
\usepackage{amscd}
\usepackage{xcolor}
\usepackage{bm}
\usepackage{enumerate}
\everymath{\displaystyle}
\usepackage{verbatim}
\usepackage{hyperref}
\usepackage{amstext}
\usepackage{latexsym}
\let\oldsqrt\sqrt
\def\sqrt{\mathpalette\DHLhksqrt}
\def\DHLhksqrt#1#2{%
\setbox0=\hbox{$#1\oldsqrt{#2\,}$}\dimen0=\ht0
\advance\dimen0-0.2\ht0
\setbox2=\hbox{\vrule height\ht0 depth -\dimen0}%
{\box0\lower0.4pt\box2}}

\allowdisplaybreaks

\newcommand{\R}{\mathbb{R}} 
\newcommand{\N}{\mathbb{N}} 

\newcommand{\e}{\varepsilon}

\newcommand{\diam}{\textnormal{diam}} 

\newcommand{\cA}{{\mathcal A}}

\newcommand{\cC}{{\mathcal C}}

\newcommand{\cL}{{\mathcal L}}

\newcommand{\cN}{{\mathcal N}}

\newcommand{\cQ}{{\mathcal Q}}

\newcommand{\cX}{{\mathcal X}}

\newcommand{\pOm}{ \partial \Omega}

\theoremstyle{definition}
\newtheorem{defi}{Definition}[section]
\newtheorem{remark}[defi]{Remark}

\theoremstyle{plain} 
\newtheorem{thm}[defi]{Theorem}

\newtheorem{lemma}[defi]{Lemma}

\theoremstyle{definition}

\numberwithin{equation}{section}

 \title{Existence and regularity results for a Neumann problem with mixed local and nonlocal diffusion}

\author{Craig Cowan$^a$\footnote{Email: \href{mailto:craig.cowan@umanitoba.ca}{craig.cowan@umanitoba.ca}}, ~~Mohammad El Smaily $^b$\footnote{Corresponding author. Email: \href{mailto:mohammad.elsmaily@unbc.ca}{mohammad.elsmaily@unbc.ca} } ~~and~~ Pierre Aime Feulefack $^b$\footnote{Email: \href{mailto:pierre.feulefack@aims-cameroon.org}{pierre.feulefack@aims-cameroon.org}}
\\
\footnotesize{$^a$ Department of Mathematics, University of Manitoba,}\\
\footnotesize{Winnipeg, MB, Canada}
\\
\footnotesize{$^b$ Department of Mathematics, University of Northern British Columbia,} \\
\footnotesize{ Prince George, BC, Canada}
}

\date{}

\begin{document}
\pdfrender{StrokeColor=black,TextRenderingMode=2,LineWidth=0.2pt}
\maketitle

\begin{abstract}
In this paper, we consider   an elliptic problem driven by a mixed local-nonlocal operator with drift  and subject to nonlocal Neumann condition. We prove  the existence and uniqueness of a solution $u\in  W^{2,p}(\Omega)$ of the considered problem with $L^p$-source function when $p$ and $s$ are in a certain range.

\end{abstract}

{\footnotesize
\textit{Keywords.}  Maximum principle, operators of mixed order with Drift, 

\vspace{.4cm}
\textit{2010 Mathematics Subject Classification}.  35A09, 35B50, 35B65, 35R11, 35J67, 47A75.  
}


\section{Introduction and main results}

In this paper, we are interested in the study of the following problem \begin{equation}\label{Eq1}
	\begin{split}
	\quad\left\{\begin{aligned}
		 Lu:=-\Delta u (x) +(-\Delta)^{s}u(x)-q(x)\cdot\nabla u(x) +a(x)u&= f(x) && \text{ in \quad $\Omega$}\\
		\frac{\partial u}{\partial \nu} &=0&& \text{ on }\quad \partial\Omega\\
		\cN_s u &=0    && \text{ on }\quad  \mathbb{R}^N\setminus \overline{\Omega},
	\end{aligned}\right.
	\end{split}
	\end{equation}
 with a mixed diffusion and a new type of Neumann boundary conditions. The new boundary condition is $\cN_s u=0$ on $\R^N\setminus \overline \Omega,$ where $\cN_s u$ --- known as the  nonlocal normal derivative of $u$ is given by
\begin{equation}\label{Normalderivative}
\cN_s u(x)=C_{N,s}\int_{\Omega} \frac{u(x)-u(y)}{|x-y|^{N+2s}}\ dy\qquad x\in \R^N\setminus\overline\Omega.
\end{equation}The diffusion term is a superposition of the classical Laplacian (local diffusion) and the fractional Laplacian $(-\Delta)^s$ for certain values of $s\in (0,1)$ that will be specified later. It is well known that the fractional Laplacian represents a  nonlocal diffusion in the medium.

 We recall that the operator $(-\Delta)^s$, with $s\in (0,1),$ stands for the fractional Laplacian and it is defined for compactly supported function $u:\R^N\to \R$ of class $\cC^2$ by
\begin{equation}
\begin{split}
 (-\Delta)^s u(x)&=C_{N,s}\lim_{\varepsilon\to 0^+}\int_{\R^N\setminus B_{\varepsilon}(x)}\frac{u(x)-u(y)}{|x-y|^{N+2s}}\ dy 
 \end{split}
\end{equation}
 with  the same  normalization constant  $C_{N,s}$ as in \eqref{Normalderivative}  given by

\begin{equation}\label{const}
C_{N,s}:= \pi^{-\frac{N}{2}}2^{2s}s\frac{\Gamma(\frac{N}{2}+s)}{\Gamma(1-s)}. ~~~
\end{equation}
We refer the reader to the results of Silvestre \cite{S07} about the spaces that $(-\Delta)^su$ belongs to according to the regularity of the function $u.$

The boundary conditions in \eqref{Eq1} consist of the classical Neumann boundary condition $\frac{\partial u}{\partial \nu}=0$ on $\partial \Omega$ ($\nu$ is the outward unit normal on $\partial\Omega$) and the nonlocal boundary condition $\cN_su=0$ (see \cite{DRV17}) on $\R^N\setminus \Omega.$ The classical Neumann condition states that there is no flux through the boundary of the domain. On the other hand, the nonlocal boundary condition $\cN_su=0$ states that if a particle is in $\R^N\setminus \overline \Omega,$ it may come back to any point $y\in \Omega$ with the probability density of
jumping from $x$ to $y$ being proportional to $|x-y|^{-N-2s}.$ A detailed description of \eqref{Eq1} is given in \cite{DV21ecology}. The condition $\cN_su=0$ is interpreted in  \cite{DV21ecology} as a condition that arises from the superposition of Brownian and L\'evy processes.

The PDE \begin{equation}\label{PDE}
-\Delta u(x) +(-\Delta)^{s}u-q(x)\cdot\nabla u(x) +a(x)u= f(x)~~\text{in}~~\Omega
\end{equation}
has been extensively studied when $q\equiv 0$ and the boundary condition is of Dirichlet type. Tat is, $u\equiv 0$ in $\R^N\setminus \Omega.$ Existence and regularity of solutions for \eqref{PDE}, as well as  maximum principles are among the results obtained in \cite{BDVV22}, \cite{BVDV21}, \cite{BDVV23}, and \cite{SVWZ22}, where the advection $q$ is absent and the boundary condition is of Dirichlet type. The authors of this paper studied \eqref{PDE} in the recent work \cite{CEF23}, where   an advection term is present and \eqref{PDE} is coupled with the Dirichlet condition $u\equiv 0$ on $\R\setminus \overline\Omega.$ 

The recent work \cite{SPV22} considers \eqref{PDE} with $q\equiv0$ and $a\equiv 0$ to provide spectral properties of the mixed diffusion operator. The work \cite{DRV17} considers a purely nonlocal diffusion and provides existence results for the problem with nonlocal Neumann conditions. It is important to note that \cite{DRV17} does not consider a PDE with a mixed diffusion and  it does not account for advection.

 \paragraph{The domain and the coefficients.} Throughout this paper, we assume that the  domain  $\Omega$ is an open bounded connected subset of $\R^N$ with smooth  boundary $\partial\Omega.$ The coefficients $q$ and $a$ are assumed to be smooth with $a \ge 0$ and not identically zero.

\paragraph{The normal derivative of $u$ on $ \pOm$.}  Our solutions will, in general, be $ C^1(\overline{\Omega})$ but the extension ($\widetilde u$ defined later) will not be sufficiently smooth.   Hence to compute  $ \partial_\nu u(x)$ on $ \pOm,$ we are using 
\[ \partial_\nu u(x) = \lim_{t \rightarrow 0^-}   \frac{ u(x_0 +t \nu(x_0) ) - u(x_0)}{t},\] where $\nu(x)$ is the unit outward normal to $\pOm$ at $x \in \pOm$.


\medskip 
We prove the following results for problem \eqref{Eq1}.

\begin{thm}    Let $\Omega$ be an open bounded set of $\R^N$ with smooth boundary and $f\in L^p(\Omega).$  Then, 

\begin{enumerate} 

\item  if  $ ~\frac{N-1}{2N}<s<\frac{1}{2}$ and $ N<p<\frac{1}{1-2s}$, problem \eqref{Eq1} admits a unique solution $u\in  W^{2,p}(\Omega).$

\item  if $p>N$ and $ \frac{1}{2}\le s< \frac{1}{2}+\frac{1}{2p},$ problem \eqref{Eq1} admits a unique solution $u\in  W^{2,p}(\Omega)$.  

\end{enumerate} 
    \end{thm}


\section{Some functional spaces}

Let $\Omega\subset \R^N$ be an open bounded set. For a vector field $q:\Omega\to \R^N$, we write $q\in L^{\infty}(\Omega)$ (resp. $q\in \cC^{0,\alpha}(\overline{\Omega})$) whenever $q_{j}\in L^{\infty}(\Omega)$ (resp. $q_j\in \cC^{0,\alpha}(\overline{\Omega})$), $j=1,2,\cdots,N.$ We denote  by $\cC^{k,\alpha}(\overline{\Omega})$, $0<\alpha\le 1$, the Banach space of functions $u\in \cC^k(\overline{\Omega})$ such that derivative of order $k$ belong to $\cC^{0,\alpha}(\overline{\Omega})$ with the norm
\[
\|u\|_{\cC^{k,\alpha}(\overline{\Omega})}:= \|u\|_{\cC^{k}(\overline{\Omega})}+\sum_{|\tau|=k}[D^{\tau}u]_{\cC^{0,\alpha}(\overline{\Omega})},
\]
where the H\"older seminorm $[\cdot]_{\cC^{0,\alpha}(\overline{\Omega})}$ is given by
\[
 [u]_{\cC^{0,\alpha}(\overline{\Omega})}=\sup_{x,y\in\overline{\Omega},x\neq y}\frac{|u(x)-u(y)|}{|x-y|^{\alpha}}
\]
and $\cC^{0,\alpha}(\overline{\Omega})$ is the Banach space of functions $u\in \cC^0(\overline{\Omega})$ which are H\"older continuous with exponent $\alpha$ in $\Omega$ and the  norm  $\|u\|_{\cC^{0,\alpha}(\overline{\Omega})}=\|u\|_{L^{\infty}(\Omega)}+[u]_{\cC^{0,\alpha}(\overline{\Omega})}$. We refer the reader to the results of Silvestre \cite{S07} about the spaces that $(-\Delta)^su$ belongs to according to the regularity of the function $u.$

If $k\in \N$, as usual we set 
$$
W^{k,p}(\Omega):=\Big\{u\in L^p(\Omega)\;:\; \text{ $D^{\alpha}u$ exists for all $\alpha\in \N^{N}$, $|\alpha|\leq k$ and $D^{\alpha} u\in L^p(\Omega)$ }\Big\}
$$
for the Banach space of $k$-times (weakly) differentialable functions in $L^p(\Omega)$. Moreover, in the fractional setting, for $s\in(0,1)$ and $p\in[1,\infty),$ we set
$$
W^{s,p}(\Omega):=\Big\{u\in L^p(\Omega)\;:\; \frac{u(x)-u(y)}{|x-y|^{\frac{n}{p}+s}}\in L^{p}(\Omega\times \Omega)\Big\}.
$$
The space $W^{s,p}(\Omega)$ is a Banach space with the norm
$$
\|u\|_{W^{s,p}(\Omega)}=\Big(\|u\|_{L^p(\Omega)}^p+\iint_{\Omega\times \Omega}\frac{|u(x)-u(y)|^p}{|x-y|^{n+s p}}\ dxdy\Big)^{\frac{1}{p}}.
$$
Next, we   set ~$\cQ:= (\R^N\times\R^N)\setminus(\R^N\setminus\Omega)^2$ and  define the space $\cX^s(\Omega)$ by
\[
\cX^s(\Omega):= \left\{ u\in L^2(\Omega): \ u|_{\Omega}\in W^{1,2}(\Omega);\ [u]_{\cX^s(\Omega)}<\infty\right\}.
\]
Here,  the corresponding Gagliardo seminorm $[\cdot]_{s}$  is given by
\[
\left([u]_{s}\right)^2:=\int_{\Omega}|\nabla u|^2\ dx+\iint_{\cQ}\frac{|u(x)-u(y)|^2}{|x-y|^{N+2s}}\ dxdy.
\]
We also define the space 
\[
\cX_K^s(\Omega):= \left\{ u\in L^2(\Omega): \ u|_{\Omega}\in W^{1,2}(\Omega);\ [u]_{s,K}<\infty\right\},
\]
where,
\[
\left([u]_{s,K}\right)^2:=\int_{\Omega}|\nabla u|^2\ dx+\int_{\Omega}\int_{\Omega}|u(x)-u(y)|^2K_{\Omega}(x,y)\ dxdy,
\]
and $K_{\Omega}:\Omega\times\Omega\to \R$ is a measurable (regional) kernel  given by
\begin{equation}\label{kernel-1}
K_{\Omega}(x,y):=\frac{1}{|x-y|^{N+2s}}+k_{\Omega}(x,y)
\end{equation}
with
\begin{equation}\label{kernel-2}
k_{\Omega}(x,y):=\int_{\R^N\setminus\Omega}\frac{1}{|x-z|^{N+2s}|y-z|^{N+2s}\displaystyle\int_{\Omega}\frac{1}{|z-z'|^{N+2s}}\ dz'}\ dz, \qquad x,y\in \Omega.
\end{equation}
Note  that the space $\cX^s(\Omega)$ $(\text{resp. ~}\cX_K^s(\Omega))$ is a Hilbert space when furnished with the scalar product $\langle u,v\rangle_{\cX^s(\Omega)}$ $(\text{resp. ~}\langle u, v\rangle_{\cX_K^s(\Omega)})$
\[
\langle u,v\rangle_{\cX^s(\Omega)} := \int_{\Omega}uv\ dx + \int_{\Omega}\nabla u\cdot \nabla v\ dx+
\iint_{\cQ}\frac{|u(x)-u(y)|^2}{|x-y|^{N+2s}}\ dxdy
\]
and the corresponding norm given $\|u\|_{\cX^s(\Omega)}=\langle u,u\rangle^{\frac{1}{2}}_{\cX^s(\Omega)}$ $(\text{resp. ~}\|u\|_{\cX^s(\Omega)}=\langle u,u\rangle^{\frac{1}{2}}_{\cX_K^s(\Omega)})$. Finally,  we define the space $\cL^1_s(\R^N)$ by 
\[
\cL^1_s(\R^N):= \big\{u:\R^N\to \R, \text{ such that $u$ is measurable~ and~ }  \|u\|_{\cL_s^1(\R^N)}<\infty\big\},
\]
where
\[
\|u\|_{\cL_1^s(\R^N)}:= \int_{\R^N}\frac{|u(y)|}{1+|y|^{N+2s}} dx.
\]

\begin{defi} Let $u\in \cC^{0,1}(\overline{\Omega})$  and  define the function
     $\widetilde{u}$ on $\R^N$ as
\begin{equation}\label{Extension-funct}
 \widetilde u(x)=
\begin{cases}
    u(x)\quad&\text{ if}\quad x\in \overline\Omega\\
    \\
u_1(x)\quad&\text{ if}\quad x\in \R^N\setminus\overline\Omega,
\end{cases}
\end{equation} 
where  \begin{equation}\label{Ext-function}
    u_1(x): 
    = \ \frac{\displaystyle\int_{\Omega}\frac{u(y)}{|x-y|^{N+2s}}\ dy}{\displaystyle\int_{\Omega}\frac{1}{|x-y|^{N+2s}}\ dy},\qquad  \text{$x\in \R^N\setminus\overline\Omega$}.
\end{equation}
\end{defi}
\begin{remark}We note that  $\cN_s\widetilde{u} (x)=0$ for all $x\in \R^N\setminus \overline{\Omega}.$
\end{remark}
We now recall the following results that lead to integration by parts in a fractional setting from \cite{AFR22}: 
\begin{lemma}[\cite{AFR22}]\label{Equiv-bilinear}
    Let $u,v:~\R^N\to \R$ be two functions such that $\cN_sv=0$  on \ $\R^N\setminus\overline\Omega$. Then
    \begin{equation}
    \begin{split}
    &\int_{\Omega}\int_{\Omega}(u(x)-u(y))(v(x)-v(y))K_{\Omega}(x,y)\ dxdy\\
    &\qquad\qquad\qquad\qquad=    C_{N,s}\iint_{\cQ}\frac{(\widetilde u(x)-\widetilde u(y))(v(x)-v(y))}{|x-y|^{N+2s}}\ dxdy
    \end{split}
    \end{equation}
\end{lemma}

From Lemma \ref{Equiv-bilinear} and \cite[Lemma 3.3]{DRV17}, we  deduce the   integration by parts formula 
\begin{equation}\label{Int-by-parts}
\begin{split}
\iint_{\cQ}\frac{(\widetilde u(x)-\widetilde u(y))(v(x)-v(y))}{|x-y|^{N+2s}}\ dxdy
&= \int_{\Omega}v(-\Delta)^s\widetilde u\ dx+\int_{\R^N\setminus\Omega}v\cN_s\widetilde u\ dx
\end{split}
\end{equation}
for $u$ and $v$ being two $\cC^2$ bounded functions in $\R^N$.

\subsection{Estimates on $\widetilde{u}$}

Let $ N<p<\infty$ and suppose $ u \in W^{2,p}(\Omega)$ and hence $u$ is $C^{0,1}(\overline{\Omega})$.  Then note that $ \widetilde{u}$ is smooth near $x $ for any $ x \notin \overline{\Omega}$.  So the only real question on the smoothness of $ \widetilde{u}$ is when $x \notin \Omega$ and $ \delta(x)=dist(x, \pOm)<1$.
  For $ x \notin \overline{\Omega},$ let $ \hat{x} \in \pOm$ be such that \[ |x- \hat{x}| = \inf_{z \in \pOm} |z-x|.\] 
 
\begin{lemma} 
\label{pointbound}
Let $\Omega$ be an open bounded set of $\R^N$ with smooth boundary and let $ N<p<\infty$ and suppose $ u \in W^{2,p}(\Omega)$ with $ \| u \|_{W^{2,p}} \le 1$.  The following estimates are all independent of $u$.  
\begin{enumerate}[(i)]
    \item For $ 0<s<1/2$ there is some $C$ such that for all  
      $~x \in \R^N\setminus \overline{\Omega}$ \ \text{ with } $\delta(x)<1$ we have 
\[ 
| \nabla u_1(x)| \le  \frac{C}{(\delta(x))^{1-2s}}.
\]  
Therefore, $ \widetilde{u} \in W^{1,q}_{loc}(\R^N)$ for all $ 1<q< \frac{1}{1-2s}$ ---  after applying the co-area formula.

\item 
If $ s= 1/2$ then  $ \widetilde{u} \in W^{1,q}_{loc}(\R^N)$  for all   $ 1\le q< \infty$.

\item  Let $ 1/2<s<1$.  Then, there is some $C>0$ such that $ | \nabla \widetilde{u}(x)| \le C$ on $ \R^N$.   

  \end{enumerate}  
\end{lemma}

\begin{proof} Let $ x \in \R^N\setminus\overline{\Omega}$ with $ \delta(x)<1$. For simplicity, we set  
\[
 F(x)= \int_\Omega |x-y|^{-N-2s} dy.
\]
Since $u\in  W^{2,p}(\Omega)$~ with   $ p> N$, we have that $u\in \cC^{0,1}(\overline{\Omega})$. It follows from  \cite[Proposition 5.2]{DRV17} and the regularity of $\Omega$  that $\widetilde u$ is continuous in  $\R^N$. Moreover,  a direct computation shows that 
\[\begin{array}{lll}
    \frac{(F(x))^2 \nabla  u_1(x) }{N+2s} &=& \int_\Omega (u(y)-u( \hat x)) |x-y|^{-N-2s} dy \int_\Omega |x-y|^{-N-2s-2} (x-y) dy \vspace{10 pt}\\
    &&+ \int_\Omega |x-y|^{-N-2s} dy \int_\Omega ( u( \hat x)-u(y)) |x-y|^{-N-2s-2} (x-y) dy.
\end{array}\] 
We now use the bound on $u$ given by $ |u(y)- u( \hat x)| \le C_0 |y- \hat x|$ to give 
\[\begin{array}{lll}
    \frac{(F(x))^2 | \nabla u_1 (x)| }{N+2s} & \le & C_0 \int_\Omega \frac{|y- \hat{x}|}{|x-y|^{N+2s}}  dy \int_\Omega \frac{1}{|x-y|^{N+2s+1}} dy \vspace{10 pt}\\
    &&+ C_0 \int_\Omega \frac{1}{|x-y|^{N+2s}}  dy \int_\Omega \frac{|y- \hat x|}{|x-y|^{N+2s+1}}  dy.
\end{array}\] 
Now note that for $ y \in \Omega$ we have $ |x- \hat x| \le |x-y|$ by the definition of $ \hat x$ and hence we have 
\[| y - \hat x| \le |y-x| + |x- \hat x| \le 2 | y-x|.\] 
It follows using the above inequality  that  
\[\begin{array}{lll}
    \frac{(F(x))^2 | \nabla  u_1(x)| }{N+2s} & \le & C_1 \int_\Omega \frac{1}{|x-y|^{N+2s-1}}  dy \int_\Omega \frac{1}{|x-y|^{N+2s+1}} dy \vspace{10 pt}\\
    &&+ C_1 \int_\Omega \frac{1}{|x-y|^{N+2s}}  dy \int_\Omega \frac{1}{|x-y|^{N+2s}}  dy.
\end{array}\]  
Noticing  that  the second term  in the right hand side is just $(F(x))^2$, we have 
\begin{equation}\label{All-case}
\frac{ | \nabla  u_1(x)|}{N+2s} \le C_1 + C_2 {\frac{\displaystyle\int_\Omega \frac{1}{|x-y|^{N+2s-1}} dy \displaystyle{\int_\Omega  \frac{1}{|x-y|^{N+2s+1}} dy}}{(F(x))^2}}.
\end{equation}
Now, it well known from \cite[Lemma 2.1]{A20} that there  are constants $C_1>0$ and $C_2>0$ such that for any $x\in \R^N\setminus\overline\Omega$, we have 
\begin{equation}\label{Estimate-Int1} 
 C_1{\min\{(\delta(x))^{-2s},(\delta(x))^{-N-2s}\}} \le F(x)\le C_2{\min\{(\delta(x))^{-2s},(\delta(x))^{-N-2s}\}}. 
\end{equation}
With the assumption $\delta(x)<1$, $x\in \R^N\setminus\overline\Omega$, this reduces to 
\begin{equation}\label{Estimate-Int2} 
C_1(\delta(x))^{-2s}\le F(x)\le C_2(\delta(x))^{-2s}.
\end{equation} 
We have the more general result that for $ \tau>0$ there is some $C_1,C_2>0$ such that for $ x \notin \overline{\Omega}$ but with $ \delta(x)$ small we have 
\begin{equation}\label{Kinling}
\frac{C_1}{(\delta(x))^\tau} \le \int_\Omega \frac{1}{|x-y|^{N+\tau}} dy \le \frac{C_2}{(\delta(x))^\tau}.
\end{equation}

We now distinguish three cases:  ~$s\in (0,1/2)$,~ $s=1/2$ ~or ~$s\in (1/2,1)$.

\paragraph{Case (i): $0<s<1/2$.}  
Set $R_0=2\diam(\Omega),$ where $\diam(\Omega)$  is the diameter of $\Omega$. Since $\Omega$ is bounded and $\delta(x)<1$, $x\in \R^N\setminus\overline\Omega$, we have that $\Omega\subset B_{2R_0+1}(x)$. It follows that
\begin{equation}\label{diameter}
\displaystyle\int_\Omega \frac{1}{|x-y|^{N+2s-1}} dy\le \displaystyle\int_{B_{2R_0+1}(0)} \frac{1}{|z|^{N+2s-1}} dz= C\int_{0}^{2R_0+1} \rho^{-2s}\ d\rho = C(2R_0+1)^{1-2s}.
\end{equation}
Putting    \eqref{Estimate-Int2}, \eqref{diameter} and \eqref{Kinling} together, we get
\begin{equation*}
  \frac{ | \nabla  u_1(x)|}{N+2s} \le C_1 + \frac{C_3}{(\delta(x))^{1-2s}}.  
\end{equation*}  We then apply the coarea formula now to get the desired result.


\paragraph{Case (ii): $s=1/2.$} We know from \eqref{All-case} that
    \[ 
    | \nabla u_1(x)| \le C_1 +  C_2\frac{   \displaystyle{\int_{\Omega} \frac{1}{|x-y|^{N+2}}} dy   \int_\Omega \frac{1}{|x-y|^N} dy}{  (F(x))^2},
    \] 
    and this gives combining \eqref{Estimate-Int2} and \eqref{Kinling},
    \[ 
    | \nabla  u_1(x)| \le C_1 + C_3 \int_\Omega \frac{1}{|x-y|^N} dy.
    \]    
    We now estimate the last term.  Since $ \Omega \subset B_{2R_0+1}(x)$ and $ x \notin \Omega$ with $ \delta(x)<1$.   Then we have 
    \[ 
    G(x)=\int_{y \in \Omega} \frac{1}{|x-y|^N} dy   \le C\int_{ \{ z: \delta(x) \le |z| \le 2R_0+1 \}}  \frac{1}{|z|^N dz } = C\ln \left(  \frac{2R_0+1}{\delta(x)} \right).
    \]  
    Then we have, after using the co-area formula,  
    \begin{eqnarray*} 
    \int_{\{ x \notin \Omega, \delta(x)<1 \}} G(x)^q dx & \le & C^q  \int_{ \{x \notin \Omega: \delta(x)<1 \} }  \left( \ln \left(  \frac{2 R_0+1}{\delta(x)} \right) \right)^q dx \vspace{10 pt}\\
    & = & C^q \int_0^1 \left(  \int_{ \{x \notin \Omega: \delta(x)=t \}} \left\{ \ln \left(  \frac{2R_0+1}{\delta(x)} \right)  \right\}^q d \sigma(x) \right) dt \vspace{10 pt}\\ 
   & = & C^q \int_0^1 \left(  \int_{ \{x \notin \Omega: \delta(x)=t \}} \left\{ \ln \left(  \frac{2R_0+1}{t} \right)  \right\}^q d \sigma(x) \right) dt 
    \end{eqnarray*} There is some $C>0$ such that $ | \{x \notin \Omega: \delta(x)=t \}| \le C$ for all $0<t<1,$ where $|A|$ refers the the $N-1$ measure of $A$.  From this we see (after doing a change of variables $ r= 1/t$ that  we have 
    \[ 
    \int_{ \{x \notin \Omega, \delta(x)<1 \}} G(x)^q dx \le C \int_1^\infty \frac{ \left( \ln((2R_0+1)r) \right)^q  }{r^2} dr,
    \] 
    and this is finite for any $ 1 \le q < \infty$. 
 Therefore,  for any  $x\in  \R^N$ and $|x|\le R$, we have 
  \begin{align*}
  \int_{B_R}|\nabla\widetilde u(x)|^q\ dx \le C_R
   \end{align*}
This shows that $ \widetilde{u} \in W^{1,q}_{loc}(\R^N)$ for all $ 1<q< \infty$.

 \paragraph{Case (iii):  ${1}/{2}<s<1$.} For $x \notin \overline{\Omega}$, we  have from \eqref{All-case} that
\[
\begin{array}{ll}
   \frac{ | \nabla  u_1(x)|}{N+2s} \le C_1 + C_2 {\frac{\displaystyle\int_\Omega \frac{1}{|x-y|^{N+2s-1}} dy \displaystyle{\int_\Omega  \frac{1}{|x-y|^{N+2s+1}} dy}}{(F(x))^2}}
    \end{array}
    \]
  Now, since   $\Omega\subset \R^N\setminus B_{\delta(x)}(x)$ for  all $x\in \Omega^c$,  we  compute for $s>1/2$,  
   \begin{align*}
       \int_{\Omega} \frac{1}{|x-y|^{N+2s-1}}\ dy \le \int_{\R^n\setminus B_{d(x)}} \frac{1}{|z|^{N+2s-1}}\ dz =  C\int_{d(x)}^{\infty} \rho^{-2s}\ d\rho = C\delta(x)^{1-2s}. 
   \end{align*}
 This combined with     \eqref{Estimate-Int2} and \eqref{Kinling} yield 
   \[
     \frac{\left|\nabla u_1 (x)\right|}{N+2s}\le C_1+C_4.
   \]
 Hence, for all  $x\in \R^N\setminus\overline\Omega$ we have
   \[
   \left|\nabla  u_1 (x)\right|\le C.
   \]

\end{proof}


\begin{lemma}\label{compactness} The following results hold: 
\begin{enumerate} \item Suppose $\frac{N-1}{2N}<s<1/2$ and $ N<p<\frac{1}{1-2s}$. Then, the mapping $ u \mapsto (-\Delta)^s \widetilde{u}$ is continuous and compact from $ W^{2,p}(\Omega)$ to $ L^p(\Omega)$. 

\item Suppose $p>N$ and $ 1/2 \le s< \frac{1}{2}+\frac{1}{2p}$.  The mapping $ u \mapsto (-\Delta)^s \widetilde{u}$ is continuous and compact from $ W^{2,p}(\Omega)$ to $ L^p(\Omega)$.  


\end{enumerate} 
    \end{lemma}

    \begin{proof} 1. For the convenience of the reader we show the continuous argument and compact argument in separate steps.   Let $ u \in W^{2,p}(\Omega)$ with $ \|u \|_{W^{2,p}(\Omega)} \le 1$ and   let $ x \in \Omega$ and then note we have 
\[ (-\Delta)^s \tilde{u}(x) = I + II + III \] where 
\[ I(u)(x) = \int_{ \{ y \in \Omega: |y-x| \le 1 \}} \frac{  \tilde{u}(x) - \tilde{u}(y) }{|x-y|^{N+2s}} dy, \] 

\[ II(u)(x) = \int_{ \{ y \in \R^N\setminus\overline{\Omega}: |y-x| \le 1 \}} \frac{  \tilde{u}(x) - \tilde{u}(y) }{|x-y|^{N+2s}} dy, \] 
\[ \text{and }~III(u)(x) = \int_{ \{ y : |y-x| \ge  1 \}} \frac{  \tilde{u}(x) - \tilde{u}(y) }{|x-y|^{N+2s}} dy. \]  Note that 
\[ | I(u)(x)| \le \int_{ \{ y \notin \Omega: |y-x| \le 1 \}}  \frac{C|x-y|}{|x-y|^{N+2s}} dy \le C_2,\] since $ s<\frac{1}{2}$.  \noindent Also, note that 
\[ | III(u)(x)| \le \int_{ \{ y : |y-x| \ge  1 \}} \frac{C}{|x-y|^{N+2s}} dy,\] since $ \widetilde{u}$ is bounded on $ \R^N$ and hence $|III(u)(x)|$ is bounded in $ \Omega$ by some $C$. We now estimate $II(u)$.  Using $ z=y-x$ we have 
\[ | II(u)(x)| \le \int_{|z| \le 1} \frac{ | \widetilde{u}(x+z)- \widetilde{u}(x)|}{ |z|^{N+2s}} dz.\] Note that 
\[ | \widetilde{u}(x+z)- \widetilde{u}(x)| \le |z| \int_0^1 | \nabla \widetilde{u}(x+ t z)| dt.\] Hence, we have 
\[\begin{array}{lll}
    \int_\Omega | II(u)(x)|^p dx & \le & \int_\Omega \left(  \int_{|z| \le 1}  \frac{1}{|z|^{N+2s-1}} \int_0^1 | \nabla \widetilde{u}(x+t z)| dt dz \right)^p dx \vspace{10 pt} \\
    & \le & C \int_\Omega \int_{|z| \le 1} \frac{1}{|z|^{N+2s-1}} \int_0^1 | \nabla \widetilde{u}(x+t z)|^p dt dz dx \quad \mbox{ \footnotesize{(Jensen's inquality applied twice)}} \vspace{10 pt}\\
    &=& C  \int_0^1 \int_{|z| \le 1} \frac{1}{|z|^{N+2s-1}} \left( \int_\Omega | \nabla \widetilde{u}(x+t z)|^p dx \right) dz dt \vspace{10 pt}\\
    & \le & C \int_0^1 \int_{|z| \le 1} \frac{1}{|z|^{N+2s-1}} \left( \int_{|x| \le R}  | \nabla \widetilde{u}(x)|^p dx \right) dz dt,
\end{array} \]
for some large $R$ and hence we have 
\[ \int_\Omega | II(u)(x)|^p dx \le C_2 \int_{|x| \le R} | \nabla \widetilde{u}(x)|^p dx.\] Combining this with the results on $I(u)$ and $III(u),$  we see that $ \Delta^s \widetilde{u} \in L^p(\Omega)$ and is continuous from $ W^{2,p}(\Omega)$.  

\noindent We now consider the compactness.   Since $ W^{2,p}(\Omega)$ is a reflexive space its sufficient to show that if $ u_m \rightharpoonup 0$ in $ W^{2,p}(\Omega)$ then $ \Delta^s \widetilde{u}_m \rightarrow 0$ in $L^p(\Omega)$.  Let $u_m \rightharpoonup 0$ in $ W^{2,p}(\Omega)$ and hence it converges to zero in $W^{1,p}(\Omega)$ and uniformly in $\Omega$.

\[\begin{array}{lll}
    \frac{(F(x))^2 \nabla \widetilde u_m(x) }{N+2s} &=& \int_\Omega (u_m(y)-u_m( \hat x)) |x-y|^{-N-2s} dy \int_\Omega |x-y|^{-N-2s-2} (x-y) dy \vspace{10 pt}\\
    &&+ \int_\Omega |x-y|^{-N-2s} dy \int_\Omega ( u_m(\hat x)-u_m(y)) |x-y|^{-N-2s-2} (x-y) dy.
\end{array}\]   From this we see that $ |\nabla \widetilde u_m(x)| \rightarrow 0 $ a.e. in $\R^N$ and note that we can use the result of Lemma \ref{pointbound} with the dominated convergence theorem to see that $ \widetilde{u}_m \rightarrow 0$ in $W^{1,p}(B_R)$ for all $ 0<R< \infty$. 

\noindent Let $ x \in \Omega$ and then note 
\[ (-\Delta)^s \widetilde u_m(x)= J_1(u_m)(x)+ III(u_m)(x)\] where 
\[ J_1(u_m)(x)= I(u_m)(x) + II(u_m)(x) = \int_{ \{ y: |y-x| \le 1 \}}  \frac{  \widetilde u_m(x) - \widetilde u_m(y)}{ |x-y|^{N+2s}} dy.\]

\noindent As before, we can write this as 
\[ \begin{array}{l}\int_\Omega | J_1(u_m)(x)|^p dx \vspace{10 pt}\\
\le C_1\int_0^1 \int_{ \{|z| \le 1 \} } \frac{1}{|z|^{N+2s-1}} \left( \int_{x \in \Omega} | \nabla  \widetilde u_m(x+tz) |^p dx \right) dz dx \vspace{10 pt}\\
\le C_2 \int_{|x|<R} | \nabla \widetilde u_m(x)|^p dx,\end{array}\] for some large $R$ and we know this goes to zero from the earlier results.  

\noindent Let $ R>1$ be big and note we have 
\[\begin{array}{lll}
    |III(u_m)(x)| & \le & \int_{ \{y: 1 \le |y-x| \le R \}} \frac{ | \widetilde u_m(x) | 
    + | \widetilde u_m(y)|}{ |x-y|^{N+2s}} dy \vspace{12 pt}\\
    &&+ \int_{ \{ y: |y-x| \ge R \} } \frac{ | \widetilde u_m(x) | + | \widetilde u_m(y)|}{ |x-y|^{N+2s}} dy \vspace{12 pt}\\
    & \le & \int_{ \{y: 1 \le |y-x| \le R \}} \frac{ | \widetilde u_m(x) | + | \widetilde u_m(y)|}{ |x-y|^{N+2s}} dy + C R^{-2s},
    \end{array}\] where $C$ is from the fact that $ | \widetilde u_m (x)| \le C_1$ on $ \R^N$ (independent of $m$).  From this we see that 
    \[\begin{array}{lll}
        \int_\Omega |III(u_m)(x)|^p dx & \le & C_p R^{-2sp} + C_p \int_\Omega \left( \int_{ 1 \le |y-x| \le R}  \frac{ |\widetilde u_m(x)| + | \widetilde u_m(y)|}{ |x-y|^{N+2s}} dy \right)^p dx\vspace{10 pt}\\
        & \le & C_p R^{-2sp}+ C_p \int_\Omega \left( 2\sup_{ |\zeta| \le R} | \widetilde u_m( \zeta)| \int_{|z| \ge 1} \frac{1}{|z|^{N+2s}} dz \right)^p dx
    \end{array}\] and note that the second term goes to zero when $m \rightarrow \infty.$ Hence, we have 
    \[ \limsup_m \int_\Omega |III(u_m)(x)|^p dx \le C_p R^{-2sp},\] and consequently $ \int_\Omega | III(u_m)(x)|^p dx \rightarrow 0$ since we can set $ R \rightarrow \infty$. \\ 

    \noindent
    2.  
    We now take $ 1/2 \le s <1$ and for these cases we split the integral in the definition of the fractional Laplacian as
\[\begin{array}{lll}
    (-\Delta)^s\widetilde u(x) & = \frac{C_{N,s}}{2}\int_{\R^N}\frac{[\widetilde u(x)-\widetilde u(x+z)]+[\widetilde u(x)-\widetilde u(x-z)]}{|z|^{N+2s}}\ dz\vspace{10pt}\\
    &=\frac{C_{N,s}}{2} \sum_{i=1}^4\int_{A_x^i}\frac{[\widetilde u(x)-\widetilde u(x+z)]+[\widetilde u(x)-\widetilde u(x-z)]}{|z|^{N+2s}}\ dz\vspace{10pt}\\
    &\qquad\quad+ \frac{C_{N,s}}{2}\int_{\{z\in\R^N:|z|>1\}}\frac{[\widetilde u(x)-\widetilde u(x+z)]+[\widetilde u(x)-\widetilde u(x-z)]}{|z|^{N+2s}}\ dz,
    \end{array}\]
 where for $i=1,\cdots,4$, the sets $A_x^i$ are defined as
\[\begin{array}{lll}
A_x^1=\{z:  |z| \le 1, x+z, x-z \in \Omega \},\vspace{10 pt}\\
A_x^2=\{z:  |z| \le 1, x+z \notin \Omega, x-z \in \Omega \},\vspace{10 pt}\\
A_x^3=\{z:  |z| \le 1, x+z \in \Omega, x-z \notin \Omega \},\vspace{10 pt}\\
A_x^4=\{z:  |z| \le 1, x+z \notin \Omega, x-z \notin \Omega \}.
\end{array}\]
We first estimate the following 
\[ 
\begin{array}{ll}
&\int_{x \in \Omega}  \left(  \int_{\{z\in \R^N:~ |z|>1\}}  \frac{ | \widetilde u(x+z) + \widetilde u(x-z) - 2 \widetilde u(x) | }{|z|^{N+2s}} dz \right)^p dx.
\end{array}
\]  First note that $  \sup_{z \in \R^N} | \widetilde{u}(z)| \le \sup_{x \in \Omega} | u(x)|$ and hence the above quantity is bounded,
\[ 
\begin{array}{ll}
&\int_{x \in \Omega}  \left(  \int_{\{z\in \R^N:~ |z|>1\}}  \frac{ | \widetilde u(x+z) + \widetilde u(x-z) - 2 \widetilde u(x) | }{|z|^{N+2s}} dz \right)^p dx\le  C \| u \|_{W^{2,p}(\Omega)}^p.
\end{array}
\]

We now estimate the integrals over $A_x^i$ for $i=1,\cdots,4.$ 

\paragraph{The $A_x^1$ term.}  Let $1/2 \le s<1$,  $ u \in W^{2,p}(\Omega)$ with $ \|u\|_{W^{2,p}} \le 1$ and let $v$ denote is $ W^{2,p}(\Omega)$ extension to all of $ \R^N$ which is compactly supported.  Using a density we assume $u,v$ smooth.  We want to estimate 
\[ 
\int_{x \in \Omega}  \left(  \int_{z \in A_x^1}  \frac{ | \widetilde u(x+z) + \widetilde u(x-z) - 2 \widetilde u(x) | }{|z|^{N+2s}} dz \right)^p dx,
\] 
and note we can replace $\widetilde u$ with $u$ in $A^1_x$ and then we can replace $u$ with $v$.  We now estimate this quantity.  First note that for $|z| \le 1$ we have 
\[ 
| v(x+z)-v(x) - \nabla v(x) \cdot z| \le |z|^2 \int_0^1 \int_0^1 |D^2 v(x+ t \tau z)| d \tau d t,
\]
and from this we see that 
\[ 
| v(x+z)+ v(x-z) - 2 v(x)| \le |z|^2 \int_0^1 \int_0^1 |D^2 v(x \pm t \tau z)| d \tau d t,
\]  
where the $\pm$ indicates there is two terms we need to consider.   Then we have 
\[
     \int_{x \in \Omega}  \left(  \int_{z \in A_x^1}  \frac{ | \widetilde u(x+z) + \widetilde u(x-z) - 2 \widetilde u(x) | }{|z|^{N+2s}} dz \right)^p dx\] 
     is bounded above by 
     \[ 
     \int_{x \in \R^N} \left( \int_{z \in A_x^1} |z|^{-N-2s+2} \left( \int_0^1 \int_0^1 |D^2 v(x \pm t \tau z)| d \tau d t \right) dz \right)^p dx 
     \] 
     and we can apply Jensen's inequality twice to get this bounded above by 
     \[
     C_1 \int_{x \in \R^N} \int_{z \in A_x^1} |z|^{-N-2s+2} \int_0^1 \int_0^1 |D^2 v(x \pm t \tau z)|^p d \tau dt dz dx \] 
     and by Fubini we see this bounded above by 
     \[ 
     C \int_0^1 \int_0^1 \int_{|z| \le 1} |z|^{-N-2s+2} \left( \int_{ \R^N} |D^2 v(x \pm t \tau z)|^p dx \right) dz d t d \tau,
     \]
     and the Extension Theorem (see \cite[Theorem 1, Page 259]{E22}) we have the term in the brackets bounded by the constant 
     \[ 
     \int_{ \R^N} |D^2 v(x)|^p dx\le C\|u\|^p_{W^{2,p}(\Omega)}
     \] and this gives us the desired bound, that is
  \[ 
\int_{x \in \Omega}  \left(  \int_{z \in A_x^1}  \frac{ | \widetilde u(x+z) + \widetilde u(x-z) - 2 \widetilde u(x) | }{|z|^{N+2s}} dz \right)^p dx\le C\|u\|^p_{W^{2,p}(\Omega)}.
\]

\paragraph{The $A_x^i$ term for $i=2,3,4$.} Note that if $z \in A_x^i$ for $ i=2,3,4,$ we must have $ |z| >\delta(x)$. 
 In what follows we will estimate 
\[\int_{x \in \Omega} \left( \int_{z \in A_x^i} \frac{  | \widetilde u(x+z)- \widetilde u(x)|}{|z|^{N+2s}} dz \right)^p dx.\] The same argument can be used to also estimate 
\[\int_{x \in \Omega} \left( \int_{z \in A_x^i} \frac{  | \widetilde u(x-z)- \widetilde u(x)|}{|z|^{N+2s}} dz \right)^p dx\] since the only fact we will use will be that $ |z| >\delta(x)$.  So to estimate the full quantity we group the three terms into the following pairings 
\[
[ \widetilde u(x+z)- \widetilde u(x) ]  + [\widetilde u(x-z) - \widetilde u(x) ]
\] 
and then estimate 
 \[ 
\int_{x \in \Omega}  \left(  \int_{z \in A_x^i}  \frac{ | \widetilde u(x+z) + \widetilde u(x-z) - 2 \widetilde u(x) | }{|z|^{N+2s}} dz \right)^p dx \qquad \text{ for }\quad i=2,\cdots,3.
\]   
We split the proof into two cases:~ $s=\frac{1}{2}$ ~and ~ $s\in (1/2,1)$.

\paragraph{The case $s=\frac{1}{2}$.}
Let $ 2 \le i \le 4$ and note that we have 
\begin{eqnarray*}
 &&\int_{x \in \Omega} \left( \int_{z \in A_x^i} \frac{  | \widetilde u(x+z)- \widetilde u(x)|}{|z|^{N+1}} dz \right)^p dx\vspace{15 pt}\\
& \le &  \int_{x \in \Omega} \left( \int_{z \in A_x^i} \frac{  \int_0^1 | \nabla \widetilde u(x+ t z) dt }{|z|^N}  dz \right)^p dx \vspace{15 pt}\\ 
 &=&  \int_{x \in \Omega} \left( \int_{z \in A_x^i} \frac{  \int_0^1 | \nabla \widetilde u(x+ t z) dt }{|z|^{N-\alpha} |z|^\alpha}  dz \right)^p dx \quad \mbox{ for some $ \alpha>0 $ small, picked later} \vspace{15 pt}\\
 & \le &  \int_{x \in \Omega} \frac{1}{(\delta(x))^{\alpha p}} \left( \int_{z \in A_x^i} \frac{  \int_0^1 | \nabla \widetilde u(x+ t z) dt }{|z|^{N-\alpha} }  dz \right)^p dx \vspace{10 pt}\\
 & \le & C \int_{x \in \Omega} \frac{1}{(\delta(x))^{\alpha p}} \left( \int_{z \in A_x^i} \frac{  \int_0^1 | \nabla \widetilde u(x+ t z)|^p dt }{|z|^{N-\alpha} }  dz \right) dx \quad \mbox{(applying  Jensen's inequality twice)} \vspace{15 pt}\\
  & \le & C \int_{x \in \Omega} \frac{1}{(\delta(x))^{\alpha p}} \left( \int_{\delta(x) < |z| \le 1} \frac{  \int_0^1 | \nabla \widetilde u(x+ t z)|^p dt }{|z|^{N-\alpha} }  dz \right) dx \vspace{15 pt} \\
  &=& C \int_0^1 \int_{|z| \le 1} \frac{1}{|z|^{N-\alpha}} \left(  \int_{ \{x \in \Omega: \delta(x) \le |z| \} }  \frac{ | \nabla \widetilde u(x+t z)|^p}{(\delta(x))^{\alpha p}} dx \right) dz dt.
\end{eqnarray*} 
\medskip

\noindent We now fix $  0 <|z| \le 1$ and $ 0<t<1$ and note for $ 1<q<\infty$ we have 
\begin{eqnarray*}
   \int_{ \{x \in \Omega: \delta(x) \le |z| \} }  \frac{ | \nabla \widetilde u(x+t z)|^p}{(\delta(x))^{\alpha p}} dx & \le & \left( \int_\Omega | \nabla \widetilde u(x+tz)|^{pq} dx \right)^\frac{1}{q} \left( \int_\Omega \frac{1}{(\delta(x))^{\alpha p q'}} dx \right)^\frac{1}{q'} 
\end{eqnarray*} and so for fixed $ q$ we can take $ \alpha>0$ small enough so that $ \alpha p q' <1$ and the internal involving the distance function is bounded and the other integral is bounded (independent of $z$ and $t$) after considering the earlier gradient bound on $ \widetilde u$.   This shows that 
\[ \int_{x \in \Omega} \left( \int_{z \in A_x^i} \frac{  | \widetilde u(x+z)- \widetilde u(x)|}{|z|^{N+1}} dz \right)^p dx
\]     is bounded.

\paragraph{The case $s\in (1/2,1)$.} Let $ 2 \le i \le 4$ and $u$ be as above. Recalling for $ z \in A_x^i$ we have $ |z|>\delta(x)$ we have 
\[\begin{array}{l}
    \int_{x \in \Omega} \left( \int_{z \in A_x^i}  \frac{  | \widetilde u(x+z)- \widetilde u(x)|}{|z|^{N+2s}} dz \right)^p dz \vspace{10 pt}\\
    \le  \int_{\Omega} \left( \int_{A_x^i} \frac{ \int_0^1 | \nabla \widetilde{u}(x+ t z)| d t}{|z|^{N+2s-1}} dz 
    \right)^p dx \vspace{10 pt} \\
    =  \int_{\Omega} \left( \int_{A_x^i} \frac{ \int_0^1 | \nabla \widetilde{u}(x+ t z)| d t}{|z|^{N+2s-1-\alpha}|z|^\alpha} dz 
    \right)^p dx \quad \mbox{ $ \alpha>0$}\vspace{10 pt}  \\
       \le   \int_{\Omega} \frac{1}{(\delta(x))^{\alpha p}} \left( \int_{A_x^i} \frac{ \int_0^1 | \nabla \widetilde{u}(x+ t z)| d t}{|z|^{N+2s-1-\alpha}} dz 
    \right)^p dx  \vspace{10 pt}\\
     \le   C \int_{\Omega} \frac{1}{(\delta(x))^{\alpha p}} \left( \int_{A_x^i} \frac{ \int_0^1 | \nabla \widetilde{u}(x+ t z)|^p d t}{|z|^{N+2s-1-\alpha}} dz 
    \right) dx \quad \mbox{(Jensen's inequality applied  twice)} 
\end{array}\] If we now assume that $ | \nabla \widetilde{u}| \le C$ then we get 
this is bounded above by 
\[ C \int_{x \in \Omega} \frac{1}{\delta(x)^{\alpha p}} \left( \int_{z \in A_x^i} \frac{1}{|z|^{N+2s-1-\alpha}} dz \right) dx,\] and since $ A_x^i \subset  \{z: |z| \le 1, |z|>\delta(x) \}$  then to have this bounded its sufficient that $ \alpha p<1 $ and $  2s-1-\alpha<0$.   Hence we see its sufficient that  $ 2s-1<\frac{1}{p}$.   The compactness proof follows the same ideas as the previous range of $s$. 

\end{proof}


The following result is a maximum principle that we will use in the proof of existence of a solution.

\begin{thm}\label{Strong-Maximum}
Let $\Omega\subset \R^N$ be an open bounded set,  $q$ and $a$ be smooth  with $a(x)\ge 0$ in $\Omega$. Suppose $u \in W^{2,p}(\Omega)$ is a solution of 
\begin{equation}\label{Eq-maximum}
	\begin{split}
	\quad\left\{\begin{aligned}
		 Lu& = 0 && \text{ in \quad $\Omega$},\\
		\frac{\partial u}{\partial \nu} &  = 0 && \text{ on }\quad \partial\Omega,\\
		\cN_s u & = 0     && \text{ on }\quad  \mathbb{R}^N\setminus \overline{\Omega}.
	\end{aligned}\right.
	\end{split}
	\end{equation}
Then, $ \widetilde u \equiv 0$ in $ \R^N$. 
 \end{thm}
\begin{proof}  Let $u$ denote the solution and, for notation, we also let $u$ denote $ \widetilde u$ outside of $ \Omega$.  Suppose $x_0 \in \overline{\Omega}$ 
such that $ u(x_0)= \inf_\Omega u$.  We first rule out $x_0 \in \pOm$.  If $ x_0 \in \pOm,$ 
then we take $ x_m \notin \overline{\Omega}$ such that $ x_m \rightarrow x_0$ as $m\rightarrow+\infty$. Using the nonlocal boundary condition we get  \[ 0 = \int_\Omega  \frac{u(y)- u(x_m)}{|y-x_m|^{N+2s}} dy.\] Passing to the limit, we see that 
\[ 0 = \int_\Omega  \frac{u(y)- u(x_0)}{|y-x_0|^{N+2s}} dy,\] which shows that $u=C=const.$ is constant in $\Omega$ and hence constant in $ \R^N$.  Then note from the equation we have $ a(x)C = f(x)$ in $ \Omega$  and this shows that $u=C \ge 0$ provided $a(x)$ not identically zero which we have assumed and hence we are done. 
We now suppose $x_0 \in \Omega$ is such that $u(x_0)= \inf_\Omega u$ and we suppose $u(x_0)<0$.  
 
We can also suppose that $u(x_0) < \inf_{\pOm} u$. Note that, from the definition of $u_1,$ we get  $u(x)> u(x_0)$ for all $x\in \R^N\setminus \overline \Omega.$

Fix $ \sigma>0$ such that $u(x_0)+ 10 \sigma < \min_{\pOm} u$ and we also assume $u(x_0)+10 \sigma<0$.   For $ \e>0$ small set $ \Omega_\e=\{x \in \Omega: \delta(x)>\e \}$ and $ \Gamma_\e=\{x \in \Omega: \delta(x)<\e \}$.  By continuity there is some $ \e_0>0$ such that 
\[ u(x_0)+ 8 \sigma < \inf_{\Gamma_{\e_0}} u.\] For $ x \notin \overline{\Omega}$ we have 

\begin{eqnarray*}
    u(x) \int_\Omega \frac{1}{|x-y|^{N+2s}} dy &=& \int_{\Gamma_{\e_0}} \frac{u(y)}{|x-y|^{N+2s}}dy
    + \int_{ \Omega \backslash\Gamma_{\e_0}} \frac{u(y)}{|x-y|^{N+2s}}dy \\
    & \ge & (u(x_0)+ 8 \sigma) \int_{\Gamma_{\e_0}} \frac{1}{|x-y|^{N+2s}}dy + u(x_0) \int_{ \Omega_\backslash\Gamma_{\e_0}} \frac{1}{|x-y|^{N+2s}}dy \\
    &=& u(x_0) \int_\Omega \frac{1}{|x-y|^{N+2s}} dy + 8 \sigma \int_{\Gamma_{\e_0}} \frac{1}{|x-y|^{N+2s}} dy
\end{eqnarray*} which gives 
\[ u(x) \ge u(x_0) +  \frac{8 \sigma \int_{\Gamma_{\e_0}} \frac{1}{|x-y|^{N+2s}} dy}{ \int_\Omega \frac{1}{|x-y|^{N+2s}} dy}.\] 
From this we can show there is some $c_{\e_0}>0$ (without loss of generality we can take $ c_{\e_0}<1$) such that $u(x) \ge u(x_0)+ 8 \sigma c_{\e_0}$ for all $ x \notin \overline{\Omega}$.

Let $ \eta$ denote a standard radial mollifier with $ \eta_\e$ the appropriately scaled function whose support is  $ \overline{B_\e}$ and set $u^\e(x) = (\eta_\e \ast u)(x)$ for $ x \in \R^N$.  For all $ 0<\e< \e_0$ we have 
\[ \inf_{ \pOm_\frac{\e}{2}} u^\frac{\e}{2} \ge \inf_{\Gamma_\e} u,\] and also we have 
$ \inf_{\Omega_\tau} u^\tau \rightarrow u(x_0)$  as $ \tau \rightarrow 0$.  From this we see there is some $ 0<\e_1<\e_0$ such that for all $ 0<\e<\e_1$ we have 
\[ \inf_{\Omega_\frac{\e}{2}} u^\frac{\e}{2} + 6 \sigma \le u(x_0) + 8 \sigma < \inf_{\Gamma_{\e_0}} u,\] but by the monotonicity of $ \e \mapsto \inf_{\Gamma_\e} u$ we have 
\[ \inf_{\Omega_\frac{\e}{2}} u^\frac{\e}{2} + 6 \sigma < \inf_{\Gamma_{\e_0}} u \le \inf_{\Gamma_\e} u\] for all $ 0<\e<\e_1$ and hence we have 
\[ \inf_{\Omega_\frac{\e}{2}} u^\frac{\e}{2} + 6 \sigma < \inf_{\pOm_{\frac{\e}{2}}} u^\frac{\e}{2},\] for $ 0<\e<\e_1$ and hence the minimum is contained in the interior of $ \Omega_\frac{\e}{2}$. 
We now want to show that $ \min_{ \Omega_\frac{\e}{2}} u^\frac{\e}{2}  \le u^\frac{\e}{2}(x) $ for all  $ x \in \R^N$.  Let $ 0<\e< \frac{\e_1}{10}$ with: \\
\begin{equation} \label{cond_1}
u^\frac{\e}{2}(x_0) < u(x_0)+8 \sigma c_{\e_0},
\end{equation} 
We consider the three cases: \\

\noindent
(i) $ x \in \Omega$ with $\delta(x)< \frac{\e}{2}$,    
(ii) $ x \notin \Omega$ with $ \delta(x)<\frac{\e}{2}$ 
 and (iii) $ x \notin \Omega$ with $\delta(x)>\frac{\e}{2}$. 

\paragraph{Case (i).} Here we have $ x \in \Omega$ with $\delta(x)< \frac{\e}{2}.$ So in this case we have 
\[ u^\frac{\e}{2}(x) = \int_{|y-x|<\frac{\e}{2}} \eta_\e(y-x) u(y) dy,\] and note the integral can be decomposed as 
\[ \int_{ |y-x|<\frac{\e}{2}, y \in \Omega}\eta_\e(y-x) u(y) dy +   \int_{ |y-x|<\frac{\e}{2}, y \notin \Omega}\eta_\e(y-x) u(y) dy \] and from this we see that 
\begin{eqnarray*}
    u^\frac{\e}{2} (x) & \ge & (u(x_0)+8 \sigma c_{\e_0}) \int_{|y-x|<\frac{\e}{2}} \eta_{\e}( y-x) dy  \\  
    & = &u(x_0)+8 \sigma c_{\e_0}  \\ 
    & > & u^\frac{\e}{2}(x_0) \\
    & \ge & \inf_{\Omega_\frac{\e}{2}} u^\frac{\e}{2}.
    \end{eqnarray*}

\paragraph{Case (ii).} In this case  we have from the definition of $\widetilde u$ that
\[ 
\begin{split}
u^\frac{\e}{2}(x)\int_{\Omega}\frac{1}{|x-y|^{N+2s}}\ dy &= \int_{\Omega}  \frac{u^\frac{\e}{2}(y)}{|x-y|^{N+2s}} dy\\
&=\int_{\Omega}  \frac{1}{|x-y|^{N+2s}}   \int_{|y-z|<\frac{\e}{2}} \eta_\e(z-y) u(z) dz dy\\
&\ge \inf_{\Omega_\frac{\e}{2}} u^\frac{\e}{2}\int_{\Omega}  \frac{1}{|x-y|^{N+2s}}\ dy,
\end{split}
\]
where we have used the computations  of Case $(i)$.

\paragraph{Case (iii).}  This follows similarly.\\

From the above we have, for small enough $\e$, that $ u^\frac{\e}{2}(x_\frac{\e}{2})= \min_{\Omega_\frac{\e}{2}} u^\frac{\e}{2}$ (some $ x_\frac{\e}{2} \in \Omega_\frac{\e}{2}$) and $ u^\frac{\e}{2}(x_\frac{\e}{2}) +6 \sigma < \inf_{\pOm_\frac{\e}{2}} u^\frac{\e}{2}$.  Also note we have (take $ \tau=\e/2$)
 $ L(u^\tau)(x)= f^\tau(x) \ge 0$ in $ \Omega_\tau$ and at $ x_\tau$ we have 
\begin{equation} \label{max_eq}
-\Delta u^\tau(x_\tau) + (-\Delta)^s u^\tau(x_\tau) + q(x_\tau) \cdot \nabla u^\tau(x_\tau) = f^\tau(x_\tau) - a(x_\tau) u^\tau(x_\tau) \ge 0.
\end{equation} But $ -\Delta u^\tau(x_\tau) \le 0$, $ \nabla u^\tau(x_\tau)=0$  and note that 
\[ (-\Delta)^s u^\frac{\e}{2}( x_\frac{\e}{2}) = \int_{y \in \R^N} \frac{  u^\frac{\e}{2}( x_\frac{\e}{2}) - u^\frac{\e}{2}(y) }{  | x_\frac{\e}{2}-y|^{N+2s}}  dy,\]  and note $ y \mapsto  u^\frac{\e}{2}( x_\frac{\e}{2}) - u^\frac{\e}{2}(y)$ is continuous in $y$ on $\R^N$, nonpositive and not identically zero.  From this we see that $(-\Delta)^s u^\frac{\e}{2}( x_\frac{\e}{2})<0$ and this contradicts (\ref{max_eq}).

\end{proof}

 \section{Proof of existence of a solution}

For $ u \in W^{2,p}(\Omega)$ and $ \gamma \in \R,$ we define 
$$L_\gamma u(x)= -\Delta u+\gamma (-\Delta)^s \tilde u(x)+a(x)u(x)+q\cdot \nabla u(x), \quad x\in \Omega$$ and we consider the family of indexed problems
\begin{equation}\label{LgammaEq}
	\begin{split}
	\quad\left\{\begin{aligned}
		L_{\gamma}  u    &= f && \text{ in \quad $\Omega$}\\
	\frac{\partial u}{\partial \nu} &=0     && \text{ on } \quad \partial\Omega.
	\end{aligned}\right.
	\end{split}
	\end{equation}
Let $\cA$ be the set 
\begin{equation}\label{setA}
\cA:=\left\{
\begin{aligned}
\gamma\in [0,1]: ~\exists C_{\gamma}>0 \text{ such that for all  }f \in L^p(\Omega),  \eqref{LgammaEq} \text{ has a  solution }\\
  u\in W^{2,p}(\Omega) \text{ such that }\|u\|_{W^{2,p}(\Omega)} \le C_{\gamma}\|f\|_{L^p(\Omega)}\qquad\qquad
 \end{aligned}
 \right\}.
\end{equation}
In \eqref{setA}, we take the constant $C_{\gamma}$ to be the smallest constant such that $\|u\|_{W^{2,p}(\Omega)}\le C_{\gamma}\|f\|_{L^p(\Omega)}$ holds for all functions $f\in L^p(\Omega)$. In other words, if $C_{\gamma}>\varepsilon>0$ then there exists  $f_{\varepsilon}\in \cC^{0,\alpha}(\overline \Omega)$ such that 
\begin{equation}\label{Smallest}
\|u\|_{W^{2,p}(\Omega)}\ge (C_{\gamma}-\varepsilon)\|f_{\varepsilon}\|_{L^p(\Omega)}.
\end{equation}

By classical theory, $ 0 \in \cA$.  Our goal is to show that $\cA$ is both open and closed and since $[0,1]$ is connected we then see that $ \cA \in \{ \emptyset, [0,1] \}$, and since its non empty, we must have $ \cA=[0,1]$ and in particular $1 \in \cA$ which corresponds to the result we are trying to prove.

\paragraph{$\cA$ is closed.} We prove $\cA$ is closed. Let $ \gamma_m \in \cA$ with $ \gamma_m \rightarrow \gamma$ and let $C_m=C_{\gamma_m}$ denote the associated constant.  We first consider the case of $ C_m$ bounded.  Let $ f \in L^p(\Omega)$ with $ \|f\|_{L^p}=1$.   Since $\gamma_m \in \cA$ there is some $ u_m \in W^{2,p}(\Omega)$ which satisfies  \eqref{LgammaEq},  with $\gamma_m$ in place of $\gamma$, and $ \|u_m \|_{W^{2,p}} \le C_m \|f\|_{L^p}=C_m$.   Note we can rewrite the problem as 
\begin{equation} \label{bring-to-r}
-\Delta u_m + au_m + q \cdot \nabla u_m = f + \gamma_m \Delta^s \widetilde u_m \quad \mbox{ in }  \Omega,
\end{equation} with $ \partial_\nu u_m=0$ on $ \pOm$.   Since $ C_m$ is bounded then we have $u_m$ bounded in $ W^{2,p}(\Omega)$ and by passing to a subsequence we can assume that $ u_m \rightharpoonup u$ in $ W^{2,p}(\Omega)$ and $ \|u \|_{W^{2,p}} \le \liminf_m \| u_m \|_{W^{2,p}} \le C_1$.  Also note by our earlier compactness result we have $ \Delta^s \widetilde u_m \rightarrow \Delta^s \widetilde u$ in $L^p(\Omega)$ and this (along with the above weak $W^{2,p}$ convergence) is sufficient convergence to pass to the limit in
(\ref{bring-to-r}). This shows that $ \gamma \in \cA$. 

We now consider the case of $C_m \rightarrow \infty$.  Then  there is some $ f_m \in L^p(\Omega)$ and $u_m \in W^{2,p}(\Omega)$ which solves $L_{\gamma_m} u_m = f_m$ in $ \Omega$ with $ \partial_\nu u_m=0$ on $ \pOm$ and 
\[ \|u_m \|_{W^{2,p}} \ge (C_m-1) \|f_m \|_{L^p}.\]  By normalizing we can assume $ \|u_m \|_{W^{2,p}(\Omega)}=1$ and hence $ \| f_m\|_{L^p} \rightarrow 0$.  By passing to subsequences we can assume that $u_m \rightharpoonup u$ in $W^{2,p}(\Omega)$ and strongly in $W^{1,p}(\Omega)$.   As before we rewrite the equation for $u_m$ by 
\begin{equation} \label{bring-to-rr}
-\Delta u_m + au_m + q \cdot \nabla u_m = f_m + \gamma_m \Delta^s \widetilde u_m \quad \mbox{ in }  \Omega,
\end{equation} with $ \partial_\nu u_m=0$ on $ \pOm$.  If $u=0$ then note the right hand side of (\ref{bring-to-rr}) converges to zero in $L^p(\Omega)$ and by standard elliptic theory we have $u_m \rightarrow 0$ in $ W^{2,p}(\Omega)$ which contradicts the normalization of $u_m$.  We now assume $ u \neq 0$.  By compactness we can pass to the limit in (\ref{bring-to-rr}) to see that $u \in W^{2,p}(\Omega) \backslash \{0\}$ satisfies $ L_{\gamma} u =0 $ in $ \Omega$ with $ \partial_\nu u=0$ on $ \pOm$ which contradicts 
Theorem \ref{Strong-Maximum}.

\paragraph{$\cA$ is open.}  Let $ \gamma_0 \in \cA$ and take $ |\e|$ small;  when $ \gamma_0 \in \{0,1 \}$ we need to restrict the sign of $ \e$.  Our goal is to show that $ \gamma=\gamma_0 + \e \in \cA$.  Fix $ f \in L^p(\Omega)$ with $ \|f\|_{L^p}=1$ and since $ \gamma_0 \in \cA$ there is some $ v_0 \in W^{2,p}(\Omega)$ which solves $ L_{\gamma_0} v_0 = f$ in $ \Omega$ with $ \partial_\nu v_0=0$ on $ \pOm$. We look for a solution of (\ref{LgammaEq}) of the form $ u= v_0+ \phi$. Writing out the details one sees we need $ \phi \in W^{2,p}(\Omega)$ to satisfy 
\begin{equation} \label{fix_st}
L_{\gamma_0} \phi = \e \Delta^s \widetilde v_0 + \e \Delta^s \widetilde \phi \quad \mbox{ in } \Omega,
\end{equation} with $ \partial_\nu \phi=0$ on $ \pOm$. Define the operator $J_\e(\phi)=\psi$ where $ \psi$ satisfies 
\begin{equation} \label{fix map}
L_{\gamma_0} \psi = \e \Delta^s \widetilde v_0 + \e \Delta^s \widetilde \phi \quad \mbox{ in } \Omega,
\end{equation} with $ \partial_\nu \psi=0$ on $ \pOm$.  We claim that for small enough $ \e$ that $J_\e$ is a contraction mapping on $ W^{2,p}(\Omega)$ and hence by the Contraction Mapping Principle there is some $ \phi \in W^{2,p}(\Omega)$ with $ J_\e(\phi)= \phi$.  From (\ref{fix_st}) one would then get a $W^{2,p}(\Omega)$ bound on $\phi$ and hence we'd get the desired bound on $ u$.   We first note that $J_\e$ is into $ W^{2,p}(\Omega)$ after noting the right hand side of (\ref{fix_st}) belongs to $ L^p(\Omega)$.   Let $ \phi_i \in W^{2,p}(\Omega)$ and $ \psi_i = J_\e(\phi_i)$ and then note we have 
\[ L_{\gamma_0} (\psi_2 - \psi_1) = \e \Delta^s ( \widetilde \phi_2 - \widetilde \phi_1) \quad \mbox{ in } \Omega,\] with the desired boundary condition.  Then we have 
\[ \| \psi_2 - \psi_1 \|_{W^{2,p}} \le C_{\gamma_0} |\e| \| \Delta^s( \widetilde \phi_2 - \widetilde \phi_1 ) \|_{L^p} \le C_{\gamma_0} |\e| C \| \phi_2 - \phi_1 \|_{W^{2,p}},\] and hence we see of $ |\e|$ small that $J_\e$ is a contraction on $W^{2,p}(\Omega)$.

\bibliographystyle{amsplain}

\end{document}